\documentclass{amsart}
\usepackage{amssymb,hyperref,url,color}
\newtheorem{theorem}{Theorem}[section]
\newtheorem{lemma}[theorem]{Lemma}
\newtheorem{reduction}[theorem]{Reduction}
\newcommand{\Sym}{\mathop{\textrm{Sym}}}
\renewcommand{\wr}{\mathop{\textrm{wr}}}
\newcommand{\CI}{\mathop{\mathrm{DCI}}}
\newcommand{\Cay}{\mathop{\mathrm{Cay}}}
\newcommand{\Aut}{\mathop{\mathrm{Aut}}}
\newcommand{\Alt}{\mathop{\textrm{Alt}}}
\newcommand{\AGL}{\mathop{\textrm{AGL}}}
\def\norm#1#2{{\bf N}_{{#1}}{{(#2)}}}
\begin{document}
\title[Dihedral DCI-groups: an addendum]
{A comment on: ``Further restrictions on the structure of finite DCI-groups''}

\author[E. Dobson]{Edward Dobson}
\address{Edward Dobson, Department of Mathematics and Statistics,
Mississippi State University,
PO Drawer MA Mississippi State, MS 39762}
\email{dobson@math.msstate.edu}

\author[J. Morris]{Joy Morris}
\address{Joy Morris, Department of Mathematics and Computer Science,
University of Lethbridge,
Lethbridge, AB. T1K 3M4. Canada}
\email{joy@cs.uleth.ca}

\author[P. Spiga]{Pablo Spiga}
\address{Pablo Spiga, Dipartimento di Matematica e Applicazioni,\newline
University of Milano-Bicocca,
 Via Cozzi 55 Milano, MI 20125, Italy}
\email{pablo.spiga@unimib.it}

\thanks{Address correspondence to P. Spiga, E-mail: pablo.spiga@unimib.it\\ The second author is supported in part by the National Science
 and Engineering Research Council of Canada.}

\begin{abstract}
A finite group $R$ is a $\CI$-group if, whenever $S$ and $T$ are subsets of $R$ with the Cayley graphs $\Cay(R,S)$ and $\Cay(R,T)$ isomorphic, there exists an automorphism $\varphi$ of $R$ with $S^\varphi=T$.

The classification of  $\CI$-groups is an open problem in the theory of Cayley graphs and is closely related to the isomorphism problem for graphs. This paper is a contribution towards this classification, as we show that every dihedral group of order $6p$, with $p\geq 5$ prime, is a $\CI$-group. This corrects and completes the proof of~\cite[Theorem~$1.1$]{LiLP} as observed by the reviewer~\cite{marston}.
\end{abstract}

\maketitle
\section{Introduction}\label{sec:intro}
Let $R$ be a finite group and let $S$ be a subset of $R$. The \textit{Cayley digraph} of $R$ with connection set $S$, denoted $\Cay(R,S)$, is the digraph with vertex set $R$ and with $(x,y)$ being an arc if and only if $xy^{-1}\in S$. Now, $\Cay(R,S)$ is said to be a \textit{Cayley isomorphic} digraph, or $\CI$-\textit{graph} for short, if whenever $\Cay(R,S)$ is isomorphic to $\Cay(R,T)$, there exists an automorphism $\varphi$ of $R$ with $S^\varphi=T$. Clearly, $\Cay(R,S)\cong \Cay(R,S^\varphi)$ for every $\varphi\in \Aut(R)$ and hence, loosely speaking, for a $\CI$-graph  $\Cay(R,S)$ deciding when a Cayley digraph over $R$ is isomorphic to  $\Cay(R,S)$ is theoretically and algorithmically elementary; that is, the solving set for $\Cay(R,S)$ is reduced to simply $\Aut(R)$ (for the definition of solving set see for example~\cite{Muz,Muz1}). The group $R$ is a $\CI$-\textit{group} if $\Cay(R,S)$ is a $\CI$-graph for every subset $S$ of $R$. Moreover, $R$ is a $\mathrm{CI}$-group if $\Cay(R,S)$ is a $\CI$-graph for every inverse-closed subset $S$ of $R$. Thus every $\CI$-group is a $\mathrm{CI}$-group.

Throughout this paper, $p$ will always denote a prime number.

In order to obtain new and severe constrains on the structure of a $\CI$-group, the authors of~\cite{LiLP} considered the problem of determining which Frobenius groups $R$ of order $6p$ are $\CI$-groups. They were in fact interested in the more specific case of Frobenius groups of order $6p$ with Frobenius kernel of order $p$; this is clear from their analysis and their proofs, but is not specified in the statement of~\cite[Theorem~$1.1$]{LiLP}. The proof of their theorem as stated is therefore incomplete,
as observed by Conder~\cite{marston}. The aim of this paper is to fix this discrepancy by completing the analysis of which Frobenius groups of order $6p$ are $\CI$-groups, hence completing the proof of ~\cite[Theorem~$1.1$]{LiLP} as the authors stated it.

An elementary computation yields that if $R$ is a Frobenius group of order $6p$ with Frobenius kernel whose order is not $p$, then $R$ is isomorphic to the alternating group on four symbols $\Alt(4)$ (and $p=2$), or to the quasidihedral group $\langle (1,2,3),(4,5,6),(2,3)(5,6)\rangle$ (and $p=3$), or to the dihedral group of order $6p$. A routine computer-assisted computation shows that $\Alt(4)$ is a $\CI$-group and  $\langle (1,2,3),(4,5,6),(2,3)(5,6)\rangle$ is not a $\CI$-group.  Moreover, as is observed in~\cite{marston}, $\langle (1,2,3),(4,5,6),(2,3)(5,6)\rangle$ is a $\mathrm{CI}$-group. Therefore in order to complete the analysis of Frobenius groups of order $6p$, we only need to consider dihedral groups of order $6p$.

\begin{theorem}\label{thrm:main}
Let $p$ be a prime number and let $R$ be the dihedral group of order $6p$. Then $R$ is a $\CI$-group if and only if $p\geq 5$, and $R$ is a $\mathrm{CI}$-group if and only if $p\geq 3$.
\end{theorem}

The structure of the paper is straightforward. In Section~\ref{smallp}, we consider the case $p\leq 5$. In Section~\ref{basictools}, we provide some preliminary definitions and our main tool. In Section~\ref{proof} we introduce some notation and we divide the proof of Theorem~\ref{thrm:main} into four  cases, which we then study in turn in Sections~\ref{caseI}--\ref{caseIV}.

\section{Small groups: $p\leq 5$}\label{smallp}

\begin{lemma}\label{l:0}
Let $p$ be a prime with $p\leq 5$ and let $R$ be the dihedral group of order $6p$. Then $R$ is a $\CI$-group if and only if $p=5$, and $R$ is a $\mathrm{CI}$-group if and only if $p\neq 2$.
\end{lemma}
\begin{proof}
The proof follows from a computer computation with the invaluable help of the algebra system \texttt{magma}~\cite{magma}. Let $R_p=\langle a,b\mid a^{3p}=b^2=(ab)^2=1\rangle$ be the dihedral group of order $6p$. Here we simply prove that $R_2$ is not a $\mathrm{CI}$-group and that $R_3$ is not a $\CI$-group.

For $p=2$, the graphs $\Cay(R_2,\{b,a^3\})$ and $\Cay(R_2,\{b,a^3b\})$ are both isomorphic to the disjoint union of three cycles of length $4$. As $a^3$ is the only central involution of $R_2$, there exists no automorphism of $R_2$ mapping $\{b,a^3\}$ to $\{b,a^3b\}$.

For $p=3$, the digraphs $\Cay(R_3,\{a,a^4,a^6,a^7\})$ and $\Cay(R_3,\{a^2,a^5,a^6,a^8\})$ are isomorphic and a computation shows that there exists no automorphism of $R_3$ mapping $\{a,a^4,a^6,a^7\}$ to $\{a^2,a^5,a^6,a^8\}$.
\end{proof}

Given that the (di)graphs we described in this proof are not connected, it is worth observing that a group $R$ is a $\mathrm{CI}$-group if and only if every pair of connected isomorphic Cayley graphs on $R$ are isomorphic via an automorphism of $R$. This is because the complement of a disconnected graph is always connected, and the property of being a $\mathrm{CI}$-graph is preserved under taking complements. A similar observation also applies to $\CI$-groups.

In view of Lemma~\ref{l:0} for the rest of this paper we may assume that $p\geq 7$.

\section{Some  basic results}\label{basictools}
Babai~\cite{babai} has proved a very useful criterion for determining when a finite group $R$ is a $\CI$-group and, more generally, when $\Cay(R,S)$ is a $\CI$-graph.

\begin{lemma}\label{lemma}Let $R$ be a finite group and let $S$ be a subset of $R$. Then $\Cay(R,S)$ is a $\CI$-graph if and only if $\Aut(\Cay(R,S))$ contains a unique conjugacy class of regular subgroups isomorphic to $R$.
\end{lemma}

Let $\Omega$ be a finite set and let $G$ be a permutation group on $\Omega$. The $2$-\textit{closure} of $G$, denoted $G^{(2)}$, is the set $$\{\pi\in \Sym(\Omega)\mid \forall (\omega,\omega')\in \Omega^2, \textrm{there exists } g_{\omega\omega'}\in G \textrm{ with }(\omega,\omega')^\pi=(\omega,\omega')^{g_{\omega\omega'}}\},$$
where $\Sym(\Omega)$ is the symmetric group on $\Omega$. Observe that in the definition of $G^{(2)}$, the element $g_{\omega\omega'}$ of $G$ may depend upon the ordered pair $(\omega,\omega')$. The group $G$ is said to be $2$-\textit{closed} if $G=G^{(2)}$.

It is easy to verify that $G^{(2)}$ is a subgroup of $\Sym(\Omega)$ containing $G$ and, in fact, $G^{(2)}$ is the smallest (with respect to inclusion) subgroup of $\Sym(\Omega)$ preserving every orbital digraph of $G$. It follows that the automorphism group of a graph is $2$-closed.
Therefore Lemma~\ref{lemma} immediately yields:

\begin{lemma}\label{babaistrong}
Let $R$ be a finite group and let $R_r$ be the right regular representation of $R$ in $\Sym(R)$. If, for every $\pi\in \Sym(R)$, the groups $R_r$ and $R_r^\pi$ are conjugate in $\langle R_r,R_r^\pi\rangle^{(2)}$, then $R$ is a $\CI$-group.
\end{lemma}
\begin{proof}
Let $S$ be a subset of $R$, and set $\Gamma:=\Cay(R,S)$ and $A:=\Aut(\Gamma)$. Observe that $R_r\leq A$ and that $A$ is $2$-closed. Let $T$ be a regular subgroup of $A$ isomorphic to $R$. Since $\langle R_r,T\rangle\leq A$, we get $\langle R_r,T\rangle^{(2)}\leq A^{(2)}=A$.

Every regular subgroup of $\Sym(R)$ isomorphic to $R$ is conjugate to $R_r$ and hence $T=R_r^\pi$, for some $\pi\in \Sym(R)$. By hypothesis, $R_r$ and $T$ are conjugate in $\langle R_r,T\rangle^{(2)}$ and so are conjugate in $A$. In particular, $A$ contains a unique conjugacy class of regular subgroups isomorphic to $R$ and Lemma~\ref{lemma} gives that $R$ is a $\CI$-group.
\end{proof}

We will use this formulation of Babai's criterion without comment in our proof of Theorem~\ref{thrm:main}.

\section{Notation and preliminary reductions}\label{proof}
Multiplication of permutations is on the right, so $\sigma\tau$ is calculated by first applying $\sigma$, and then $\tau$.  For the rest of this paper we let $R$ be the dihedral group of order $6p$ and we let $\Omega:=\{1,\ldots,6p\}$. Using Lemma~\ref{l:0}, we may assume that $p\ge 7$ in the proof of Theorem~\ref{thrm:main}. In what follows, we identify $R$ with a regular subgroup of $\Sym(\Omega)$ isomorphic to $R$, that is, $R$ acts regularly on $\Omega$. Let $\pi\in \Sym(\Omega)$ and set $G:=\langle R,R^\pi\rangle$. In view of Lemma~\ref{babaistrong}, Theorem~\ref{thrm:main} will follow by proving that $R$ is conjugate to $R^\pi$ via an element of $G^{(2)}$.

Let $R_p$ denote the Sylow $p$-subgroup of $R$, let $P$ be a Sylow $p$-subgroup of $G$ with $R_p\leq P$ and let $T$ be a Sylow $p$-subgroup of $\Sym(\Omega)$ with $P\leq T$. From Sylow's theorems, replacing $R^\pi$ by a suitable $G$-conjugate, we may assume that $R_p^\pi\leq P$. Observe that, as $p\geq 7$, the group $T$ is elementary abelian of order $p^6$. Since $R_p$ and $R_p^\pi$ are acting semiregularly, their orbits on $\Omega$ must be equal to the orbits of $T$.

Since $R_p$  is the unique Sylow $p$-subgroup of $R$, we see that $R$ admits a unique system of imprimitivity $\mathcal{C}$  with blocks of size $p$, namely $\mathcal{C}$ consists of the $R_p$-orbits on $\Omega$. Similarly, $R^\pi$ admits a unique system of imprimitivity with blocks of size $p$, namely $\mathcal{C}^\pi$, and the system of imprimitivity $\mathcal{C}^\pi$ consists of the $R_p^\pi$-orbits on $\Omega$. Since each of these is equal to the orbits of $T$ on $\Omega$, we have $\mathcal{C}=\mathcal{C}^\pi$, and $\mathcal{C}$ is $R$- and $R^\pi$-invariant. As $G=\langle R,R^\pi\rangle$, we get that $\mathcal{C}$ is also $G$-invariant. Therefore, $G$ is conjugate to a subgroup of $\Sym(p)\wr\Sym(6)$. Similarly, since $\mathcal{C}$ is $\pi$-invariant, $\pi$ is conjugate to an element in $\Sym(p)\wr\Sym(6)$.

We can use this structure to decompose the set $\Omega$ as $\Delta\times \Lambda$ with $|\Delta|=p$ and $|\Lambda|=6$. We identify $\Omega$ with $\Delta\times \Lambda$, $\Delta$ with $\{1,\ldots,p\}$ and $\Lambda$ with $\{1,\ldots,6\}$.
Write $W:=\Sym(\Delta)\wr \Sym(\Lambda)$ and $B:=\Sym(\Delta)^6$ the base group of $W$. Then for $\sigma\in \Sym(\Lambda)$, $(y_1,\ldots,y_6)\in B$, and $(\delta,\lambda)\in\Delta\times\Lambda$, we have $$(\delta,\lambda)^{\sigma}=(\delta,\lambda^\sigma) \,\textrm{   and   }\, (\delta,\lambda)^{(y_1,\ldots,y_6)}=(\delta^{y_\lambda},\lambda),$$
and $W=\{\sigma(y_1,\ldots,y_6)\mid\sigma\in\Sym(\Lambda),(y_1,\ldots,y_6)\in B\}$. Observe that under this identification the system of imprimitivity $\mathcal{C}$ is $\{\Delta_1,\ldots,\Delta_6\}$ where $\Delta_\lambda=\Delta\times \{\lambda\}$ for every $\lambda\in \Lambda$.

 Let $K$ be the kernel of the action of $G$ on $\mathcal{C}$, that is, $K=B\cap G$. Clearly, $RK/K$ and  $R^\pi K/K$ are regular subgroups of $\Sym(\Lambda)$ isomorphic to $\Sym(3)$.
A direct inspection in $\Sym(\Lambda)$ shows that if $A$ and $B$ are regular subgroups of $\Sym(\Lambda)$ isomorphic to $\Sym(3)$, then either $B$ is conjugate to $A$ via an element of $\langle A,B\rangle$, or $\langle A,B\rangle=A\times B$. Summing up and applying this observation to $G/K$, we obtain the following reduction.
\begin{reduction}\label{reduction1}
{\rm
We have $$G\leq W\quad\textrm{and}\quad \pi\in W,$$
and  (replacing $G$ by a suitable $W$-conjugate) either
\begin{equation}\label{red1eq1}\frac{G}{K}=\frac{RK}{K}=\frac{R^\pi K}{K}=\langle (1,2,3)(4,5,6),(1,4)(2,6)(3,5)\rangle,
\end{equation} or
\begin{eqnarray}\label{red1eq2}
\frac{G}{K}&=&\frac{RK}{K}\times \frac{R^\pi K}{K},\\\nonumber
 RK/K&=&\langle (1,2,3)(4,5,6),(1,4)(2,6)(3,5),\\\nonumber R^\pi K/K&=&\langle(1,2,3)(4,6,5),(1,4)(2,5)(3,6)\rangle.
\end{eqnarray}

A moment's thought gives that in case~\eqref{red1eq1} we may assume that $\pi\in B$ and in case~\eqref{red1eq2} we may assume that $\pi=(5,6)y$ with $y\in B$. Write $\pi:=\sigma(y_1,\ldots,y_6)$ with $\sigma=1$ or $\sigma=(5,6)$ depending on whether case~\eqref{red1eq1} or~\eqref{red1eq2} is satisfied. Set $y:=(y_1,\ldots,y_6)$.}
\end{reduction}

Let $c$ be the cycle $(1,2,\ldots,p)$ of length $p$ of $\Sym(\Delta)$. Set $$r_1:=(c,c,c,c,c,c),\, r_2:=(1,2,3)(4,5,6) \textrm{ and }r_3:=(1,4)(2,6)(3,5).$$ Replacing $G$ by a suitable $W$-conjugate, we may assume that
\begin{equation}\label{eq0}
R_p=\langle r_1\rangle\textrm{  and  }R=\langle r_1,r_2,r_3\rangle.
\end{equation}
Clearly, $\norm {\Sym(\Delta)}{\langle c\rangle}\cong \AGL_1(p)$ and hence $\norm{\Sym(\Delta)}{\langle c\rangle}=\langle c,\alpha\rangle=\langle c\rangle\rtimes \langle\alpha\rangle$, where $\alpha$ is a permutation fixing $1$ and acting by conjugation on $\langle c\rangle$ as an automorphism of order $p-1$.

As $R_p\leq T$, we see that $T$ is generated by $c_1,c_2,\ldots,c_6$ where
$$c_1:=(c,1,1,1,1,1),c_2:=(1,c,1,1,1,1),\ldots,c_6:=(1,1,1,1,1,c).$$
Since $R_p^\pi\leq T$ and since $R_p^\pi$ is semiregular, we obtain $$R_p^\pi=\langle (c^{\ell_1},c^{\ell_2},c^{\ell_3},c^{\ell_4},c^{\ell_5},c^{\ell_6})\rangle,$$ with $\ell_1=1$ and for some $\ell_2,\ldots,\ell_6\in \{1,\ldots,p-1\}$.

Now $r_1^\pi=(c^{y_1},c^{y_2},c^{y_3},c^{y_4},c^{y_5},c^{y_6})\in R_p^\pi$ and hence there exists $\ell\in \{1,\ldots,p-1\}$ with
$c^{y_\lambda}=c^{\ell_\lambda\ell}$, for every $\lambda\in \Lambda$. Thus $y_\lambda\in\norm{\Sym(\Delta)}{\langle c\rangle}=\langle c,\alpha\rangle$ and  $y_\lambda=c^{u_\lambda}\alpha^{v_\lambda}$ for some $u_\lambda\in \{0,\ldots,p-1\}$ and $v_\lambda\in \{0,\ldots,p-2\}$. It follows that
\begin{eqnarray}\label{eq3}
\pi&=&\sigma(c^{u_1}\alpha^{v_1},c^{u_2}v^{\alpha_2},\ldots,c^{u_6}\alpha^{v_6})\in \langle c,\alpha\rangle\wr\Sym(\Lambda),\\\nonumber
G&\leq& \langle c,\alpha\rangle\wr \Sym(\Lambda).
\end{eqnarray}
Now $r_1\in R\leq G$, and hence replacing $\pi$ by $r_1^{-u_1}\pi$, we may assume that $u_1=0$. Furthermore, $(\alpha,\alpha,\alpha,\alpha,\alpha,\alpha)\in \norm{\Sym(\Omega)}R$, and hence replacing $\pi$ by $(\alpha,\ldots,\alpha)^{-v_1}\pi$, we may assume that $v_1=0$.

As $\langle c,\alpha\rangle\wr\Sym(\Lambda)$ has a normal Sylow $p$-subgroup, we get  $P\unlhd G$ and $K/P$ is isomorphic to a subgroup of $\langle\alpha\rangle\times\langle\alpha\rangle\times\langle\alpha\rangle\times\langle\alpha\rangle\times\langle\alpha\rangle\times\langle\alpha\rangle$.

\smallskip

Next we define an equivalence relation $\equiv$ on $\Omega$. We say that $\omega\equiv \omega'$ if $P_{\omega}=P_{\omega'}$. Since $P\unlhd G$, we see that $\equiv$ is $G$-invariant. Moreover, since $P$ is abelian, we get that $P$ acts regularly on each of its orbits and hence $\omega\equiv \omega'$ for every $\omega$ and $\omega'$ in the same $P$-orbit. This shows that $\equiv$ defines a system of imprimitivity $\mathcal{E}$ for $G$ coarser than $\mathcal{C}$. In particular, $\equiv$ consists of either $1$, $2$, $3$ or $6$ equivalence classes.

There is an equivalent definition of $\equiv$. Given  $\omega\in\Delta_\lambda$ and $\omega'\in\Delta_{\lambda'}$, we have $\omega\equiv \omega'$ whenever, for every $\rho\in P$, $\rho\vert_{\Delta_\lambda} = 1$ if and only if $\rho\vert_{\Delta_{\lambda'}} = 1$ (or equivalently, $\rho\vert_{\Delta_\lambda}$ is a $p$-cycle if and only if $\rho\vert_{\Delta_{\lambda'}}$ is a $p$-cycle).

We will use the following lemma repeatedly.

\begin{lemma}\label{tedslemma}
For every $\rho\in K$ and for every $E\in \mathcal{E}$, the permutation $\rho_E:\Omega\to \Omega$, fixing $\Omega\setminus E$ pointwise and acting on $E$ as $\rho$ does, lies in  $G^{(2)}$.
\end{lemma}
\begin{proof}
This is Lemma~$2$ in~\cite{Dobson1995}.  (We remark that~\cite[Lemma 2]{Dobson1995} is only stated for graphs, but the result holds for each orbital digraph of $G$, and hence for $G^{(2)}$.)
\end{proof}

With all of this notation at our disposal we are ready to prove  Theorem~\ref{thrm:main} with a case analysis depending on the number of $\equiv$-equivalence classes.

\section{Case I: $\equiv$ has only one equivalence class}\label{caseI}

Here, $P_\omega=P_{\omega'}$ for every $\omega,\omega'\in\Omega$, hence $P$ acts semiregularly on $\Omega$ and $|P|=p$. It follows that $P=R_p=R_p^\pi$. In particular, $\ell_1=\cdots =\ell_6=1$ and $v_1=\cdots=v_6=0$. Therefore $\pi=\sigma(c^{u_1},c^{u_2},c^{u_3},c^{u_4},c^{u_5},c^{u_6})$ with $u_1=0$.

Suppose that $\sigma=1$. Since $r_2,r_2^\pi\in G$, we have
$$r_2^{-1}(r_2)^\pi=(
c^{-u_3+u_1},
c^{-u_1+u_2},
c^{-u_2+u_3},
c^{-u_6+u_4},
c^{-u_4+u_5},
c^{-u_5+u_6}
)\in P$$
%c^{-u_2}\beta,
%c^{-u_3+u_2}\beta,
%\beta^{-1}c^{u_3}\beta,
%\beta^{-1}c^{-u_5+u_4}\beta,
%\beta^{-1}c^{-u_6+u_5}\beta,
%\beta^{-1}c^{-u_4+u_6}\beta
%)\in P$$
and hence $-u_3+u_1=-u_1+u_2=-u_2+u_3=-u_6+u_4=-u_4+u_5=-u_5+u_6$. This gives $u_1=u_2=u_3=0$ and $u_4=u_5=u_6$. Write $u:=u_4$. A similar computation gives
$$r_3^{-1}(r_3)^\pi=(
c^{-u},
c^{-u},
c^{-u},
c^{u},
c^{u},
c^{u}
)\in P.$$
Thus $u=-u$ and hence $u=0$. Therefore $\pi=1$ and $R^\pi=R$. It follows that $R$ is conjugate to $R^\pi$ via the identity element of $G^{(2)}$.

\smallskip

Suppose that $\sigma=(5,6)$. Since $r_2,r_2^\pi\in G$, we have
$$r_2^{-1}(r_2)^\pi=(4,5,6)(
c^{-u_3+u_1},
c^{-u_1+u_2},
c^{-u_2+u_3},
c^{-u_5+u_4},
c^{-u_6+u_5},
c^{-u_4+u_6}
)\in G$$
and by taking the $3^{\mathrm{rd}}$ power we get $(
c^{3(-u_3+u_1)},
c^{3(-u_1+u_2)},
c^{3(-u_2+u_3)},1,1,1)\in P$.
Thus $3(-u_3+u_1)=3(-u_1+u_2)=3(-u_2+u_3)=0$ and since $u_1=0$, we have $u_1=u_2=u_3=0$. Moreover $$r_2(r_2)^{\pi}=(1,3,2)(1,
1,1,
c^{-u_5+u_4},
c^{-u_6+u_5},
c^{-u_4+u_6}
)\in G$$ and  by taking the $3^{\mathrm{rd}}$ power we get $(1,1,1,
c^{3(-u_5+u_4)},
c^{3(-u_6+u_5)},
c^{3(-u_4+u_6)})\in P$. Thus $3(-u_5+u_4)=3(-u_6+u_5)=3(-u_4+u_6)=0$ and hence $u_4=u_5=u_6$. Write $u:=u_4$.
Now $$r_3^{-1}(r_3)^\pi=(2,3)(5,6)(
c^{-u},
c^{-u},
c^{-u},
c^{u},
c^{u},
c^{u}
)\in G$$
and by taking the $2^{\mathrm{nd}}$ power we get $(
c^{-2u},
c^{-2u},c^{-2u},
c^{2u},c^{2u},c^{2u})\in P$. Thus $2u=-2u$, and hence $u=0$. It follows that $\pi=\sigma=(5,6)$ and
$$G=\langle R,R^\pi\rangle=\langle r_1,(1,2,3)(4,5,6),(1,4)(2,6)(3,5),(1,2,3)(4,6,5),(1,4)(2,5)(3,6)\rangle.$$

We claim that $\pi\in G^{(2)}$, from which the proof of this case follows. First observe that $(1,2,3)(4,5,6)(1,2,3)(4,6,5)=(1,3,2) \in G$. Also $r_3^{-1}r_3^\pi=(2,3)(5,6)\in G$, and hence (conjugating by the elements of $\langle(1,3,2)\rangle$), we see that $(1,2)(5,6)$ and $(1,3)(5,6)$ belong to $G$. Next, let $\omega=(\delta,\lambda)$ and $\omega'=(\delta',\lambda')$ be in $\Omega$. If $\lambda,\lambda'\notin\{5,6\}$, then $(\omega,\omega')^{\pi}=(\omega,\omega')^{g_{\omega\omega'}}$ with $g_{\omega\omega'}=1$. If $\lambda,\lambda'\in \{5,6\}$, then $(\omega,\omega')^{\pi}=(\omega,\omega')^{g_{\omega\omega'}}$ with $g_{\omega\omega'}=(1,2)(5,6)$. Finally, suppose that only one of $\lambda,\lambda'$ lies in $\{5,6\}$. Let $\lambda''$ be the element of $\{\lambda,\lambda'\}\cap\{1,2,3,4\}$ and let $g_{\omega\omega'}$ be in $\{(1,2)(5,6),(1,3)(5,6),(2,3)(5,6)\}$ fixing the block $\Delta_{\lambda''}$ pointwise. Then  $(\omega,\omega')^{\pi}=(\omega,\omega')^{g_{\omega\omega'}}$.

\section{Case II: $\equiv$ has six  equivalence classes}\label{caseII}

Since $\equiv$ has six equivalence classes, for every two distinct $\lambda,\lambda'\in\Lambda$, there exists an element $q\in P$ with $q$ fixing $\Delta_{\lambda}$ pointwise and acting as the cycle $c$ on $\Delta_{\lambda'}$. From this it  follows that $P^{(2)}=T$. Next, from $T\leq G^{(2)}$, it follows that if $\gamma:\Omega\to\Omega$ is a permutation with the property that for each $\lambda \in \Lambda$, we have
\begin{itemize}
\item $\Delta_{\lambda}^\gamma=\Delta_\lambda$ and \item $\gamma\vert_{\Delta_{\lambda}}=g_\lambda\vert_{\Delta_{\lambda}}$ for some $g_\lambda\in G$ fixing $\Delta_\lambda$ setwise,
\end{itemize}
 then $\gamma \in G^{(2)}$.

As $ T=P^{(2)}\leq  G^{(2)}$, replacing $\pi$ by $g^{-1}\pi$ for a suitable $g\in T$, we may assume that $u_1=u_2=\cdots=u_6=0$.

For $2 \le \lambda \le 6$, let $g_\lambda$ be the element of $R$ that maps $(1,1)$ to $(1,\lambda)$ (so $g_2=r_2$, etc.). Define $\gamma:\Omega\to \Omega$ by $\gamma\vert_{\Delta_1}={\rm id}\vert_{\Delta_1},$ and for $2 \le \lambda \le 6$, $$\gamma\vert_{\Delta_\lambda}=\left((g_{\lambda^\sigma}^\pi)^{-1}g_\lambda\right)\vert_{\Delta_\lambda}.$$ By the observations we made in the first paragraph of this case, $\gamma \in G^{(2)}$.
Careful computations show that $(r_1^\pi)^{\gamma}=r_1$. Thus, $(R_p^\pi)^\gamma=R_p$. We now see that after conjugating $R^\pi$ by $\gamma$ we are in Case I and can complete the proof as before.

\section{Case III: $\equiv$ has two equivalence classes}\label{caseIII}

The $\equiv$-equivalence classes are blocks of imprimitivity for $G$ of size $3p$ and are a union of $P$-orbits. The only system of imprimitivity for $G/K$  with blocks of size $3$ is $\{\{1,2,3\},\{4,5,6\}\}$. Therefore the two $\equiv$-equivalence classes are $\Delta_1\cup\Delta_2\cup\Delta_3$ and $\Delta_4\cup\Delta_5\cup\Delta_6$.  By Lemma \ref{tedslemma} applied to $\rho = r_1$, $(c,c,c,1,1,1),(1,1,1,c,c,c)\in G^{(2)}$.

%{\color{red} We claim that $P=\langle (c,c,c,1,1,1),(1,1,1,c,c,c)\rangle=P^{(2)}$. Since $\Delta_1\cup\Delta_2\cup\Delta_3$ is a $\equiv$-equivalence class, $P$ contains an element of the form $q=(x,y,z,1,1,1)$ with $x,y,z\in\langle c\rangle$ and $x\neq 1$. Now $\overline{x}:=(x,x,x,x,x,x)\in R\leq G$ and hence $q\overline{x}^{-1}=(1,yx^{-1},zx^{-1},x^{-1},x^{-1},x^{-1})\in P$. Since  $\Delta_1\cup\Delta_2\cup\Delta_3$ is a $\equiv$-equivalence class, we get $yx^{-1}=zx^{-1}=1$, that is, $x=y=z$. From this it easily follows that  $P=\langle (c,c,c,1,1,1),(1,1,1,c,c,c)\rangle$. Now a routine computation shows that $P=P^{(2)}$.}

Replacing $\pi$ by $g^{-1}\pi$ for a suitable $g\in G^{(2)}$, we may assume that $u_4=0$.
As $R_p^\pi\leq P$, we get $\ell_1=\ell_2=\ell_3$ and $\ell_4=\ell_5=\ell_6$. It follows that $v_1=v_2=v_3=0$ and $v_4=v_5=v_6$. Write $\beta:=\alpha^{v_4}$. Therefore $\pi=\sigma(1,c^{u_2},c^{u_3},\beta,c^{u_5}\beta,c^{u_6}\beta)$.

Suppose that $\sigma=1$. We have
$$r_2^{-1}(r_2)^\pi=(
c^{-u_3},
c^{u_2},
c^{-u_2+u_3},
\beta^{-1}c^{-u_6}\beta,
\beta^{-1}c^{u_5}\beta,
\beta^{-1}c^{-u_5+u_6}\beta
)\in P$$
and hence $-u_3=u_2=-u_2+u_3$ and $-u_6=u_5=-u_5+u_6$. This gives $u_2=u_3=0$ and $u_5=u_6=0$, that is, $\pi=(1,1,1,\beta,\beta,\beta)$. A similar computation gives
$$r_3^{-1}(r_3)^\pi=(
\beta^{-1},
\beta^{-1},
\beta^{-1},
\beta,
\beta,
\beta
)\in K.$$
Applying Lemma~\ref{tedslemma} with $E:=\Delta_4\cup\Delta_5\cup\Delta_6$ and  $\rho:=r_3^{-1}(r_3)^\pi$, we get $(1,1,1,\beta,\beta,\beta)\in G^{(2)}$, that is, $\pi\in G^{(2)}$, from which the proof follows.

Suppose that $\sigma=(5,6)$.  Since $r_2,r_2^\pi\in G$, we have
$$r_2^{-1}(r_2)^\pi=(4,5,6)(
c^{-u_3},
c^{u_2},
c^{-u_2+u_3},
\beta^{-1}c^{-u_5}\beta,
\beta^{-1}c^{-u_6+u_5}\beta,
\beta^{-1}c^{u_6}\beta
)\in G$$
and by taking the $3^{\mathrm{rd}}$ power we get $(
c^{-3u_3},
c^{3u_2},
c^{3(-u_2+u_3)},1,1,1)\in P$.
Thus $-3u_3=3u_2=3(-u_2+u_3)$ and hence $u_1=u_2=u_3=0$. Moreover $$r_2(r_2)^{\pi}=(1,3,2)(1,
1,1,
\beta^{-1}c^{-u_5}\beta,
\beta^{-1}c^{-u_6+u_5}\beta,
\beta^{-1}c^{u_6}\beta
)\in G$$ and  by taking the $3^{\mathrm{rd}}$ power we get $(1,1,1,
\beta^{-1}c^{-3u_5}\beta,
\beta^{-1}c^{3(-u_6+u_5)}\beta,
\beta^{-1}c^{3u_6}\beta)\in P$. Thus $-3u_5=3(-u_6+u_5)=3u_6$ and hence $u_4=u_5=u_6=0$. Thus $\pi=(5,6)(1,1,1,\beta,\beta,\beta)$ and
$r_2^{-1}r_2^\pi=(4,6,5)\in G$. This gives $\langle (1,2,3),(4,5,6)\rangle\leq G$.

Now $$r_3^{-1}(r_3)^\pi=(2,3)(5,6)(
\beta^{-1},
\beta^{-1},
\beta^{-1},
\beta,
\beta,
\beta
)\in G.$$
Call this element $\hat{g}_1$. As $(1,2,3)\in G$, we have $$\hat{g}_2:=\hat{g}_1^{(1,2,3)}=(1,3)(5,6)(
\beta^{-1},
\beta^{-1},
\beta^{-1},
\beta,
\beta,
\beta
)\in G$$ and $$\hat{g}_3:=\hat{g}_1^{(1,3,2)}=(1,2)(5,6)(
\beta^{-1},
\beta^{-1},
\beta^{-1},
\beta,
\beta,
\beta
)\in G.$$

We claim that $\pi\in G^{(2)}$, from which the proof of this case immediately follows.  Let $\omega=(\delta,\lambda)$ and $\omega'=(\delta',\lambda')$ be in $\Omega$. If $\lambda,\lambda'\in\{1,2,3\}$, then $(\omega,\omega')^{\pi}=(\omega,\omega')^{g_{\omega\omega'}}$ with $g_{\omega\omega'}=1$. If $\lambda,\lambda'\in\{4,5,6\}$, then $(\omega,\omega')^{\pi}=(\omega,\omega')^{g_{\omega\omega'}}$ with $g_{\omega\omega'}=\hat{g}_1$.  Finally, suppose that only one of $\lambda,\lambda'$ lies in  $\{1,2,3\}$. Without loss of generality we may assume that $\lambda\in \{1,2,3\}$ and $\lambda'\in\{4,5,6\}$. Thus $\omega^\pi=(\delta,\lambda)^\pi=(\delta,\lambda)$ and $\omega'^\pi=(\delta',\lambda')^\pi=(\delta'^\beta,\lambda'^{(5,6)})$. Since $\langle c\rangle$ is transitive on $\Delta$, there exists $x\in \langle c\rangle$ with $\delta^x=\delta^{\beta^{-1}}$. Set $g_{\omega\omega'}:=\hat{g}_{\lambda}(x,x,x,1,1,1)^{-1}$ and observe that $g_{\omega\omega'}\in G$. By construction, we have $(\omega,\omega')^{\pi}=(\omega,\omega')^{g_{\omega\omega'}}$.

\section{Case IV: $\equiv$ has three equivalence class}\label{caseIV}

Observe that the $\equiv$-equivalence classes are blocks of imprimitivity for $G$ of size $2p$ and are union of $P$-orbits. In case~\eqref{red1eq2} of Reduction~\ref{reduction1}, the group $G/K$ has no system of imprimitivity with blocks of size $2$ and hence this case cannot arise. Therefore only case~(\ref{red1eq1}) can happen, that is, $\sigma=1$.

The group $G/K\cong\langle (1,2,3)(4,5,6),(1,4)(2,6)(3,5)\rangle$ has three subgroups of order $2$ and hence $G/K$ has three systems of imprimitivity   with blocks of size $2$, namely  $\{\{1,4\},\{2,6\},\{3,5\}\}$, $\{\{1,5\},\{2,4\},\{3,6\}\}$ and $\{\{1,6\},\{2,5\},\{3,4\}\}$. Without loss of generality we may assume that the three $\equiv$-equivalence classes are $\Delta_1\cup\Delta_4$, $\Delta_2\cup\Delta_6$ and $\Delta_3\cup\Delta_5$.

Applying Lemma~\ref{tedslemma} with $\rho:=r_1$ and with $E\in \{\Delta_1\cup\Delta_4,\Delta_2\cup\Delta_6,\Delta_3\cup\Delta_5\}$, we get
$$\hat{P} := \langle (c,1,1,c,1,1),(1,c,1,1,1,c),(1,1,c,1,c,1)\rangle \le G^{(2)}.$$ \noindent Replacing $\pi$ by $g^{-1}\pi$ for a suitable $g\in \hat{P}$, we may assume that $u_2=u_3=0$. Furthermore,
as $R_p^\pi\leq P$, we get $\ell_1=\ell_4$, $\ell_2=\ell_6$ and $\ell_3=\ell_5$. It follows that $v_1=v_4=0$ and $v_2=v_6$ and $v_3=v_5$. Write $\beta:=\alpha^{v_2}$ and $\gamma:=\alpha^{v_3}$. Therefore $\pi=(1,\beta,\gamma, c^{u_4} ,c^{u_5}\gamma,c^{u_6}\beta)$.

We have
$$r_3^{-1}(r_3)^\pi=(c^{-u_4},\beta^{-1}c^{-u_6}\beta,\gamma^{-1}c^{-u_5}\gamma,c^{u_4},\gamma^{-1}c^{u_5}\gamma,\beta^{-1}c^{u_6}\beta)\in P$$
and hence $-u_4=u_4$, $-u_5=u_5$ and $-u_6=u_6$. Thus $u_4=u_5=u_6=0$ and $\pi=(1,\beta,\gamma,1,\gamma,\beta)$. Similarly, we have
$$r_2^{-1}(r_2)^\pi=(\gamma^{-1},\beta,\beta^{-1}\gamma,\beta^{-1},\gamma,\gamma^{-1}\beta)\in K.$$
Call this element $g$. As $\Delta_1\cup\Delta_4$ is a $\equiv$-equivalence class, $\gamma^{-1} = \beta^{-1}$ and hence $\pi = (1,\beta,\beta,1,\beta,\beta)$ and $g = (\beta^{-1},\beta,1,\beta^{-1},\beta,1)$.  Applying Lemma~\ref{tedslemma} with $\rho:=g$ and $E:=\Delta_2\cup\Delta_5$, we get $g':=(1,\beta,1,1,\beta,1)\in G^{(2)}$. Thus $g'':=(g')^{r_2}=(1,1,\beta,1,1,\beta)\in G^{(2)}$ and $\pi =g'g''\in G^{(2)}$, from which the proof follows.

\end{document}